\documentclass[11pt]{article}
\usepackage{sigsam, amsmath, amsthm}
\newtheorem{theorem}{Theorem}
\newtheorem{lemma}[theorem]{Lemma}
\newtheorem*{goal}{Goal}
\numberwithin{theorem}{section}

\issue{To Appear}
\articlehead{ISSAC 2024 abstracts}
\titlehead{Complex reflection groups as differential Galois groups}
\authorhead{Authors}
\authorhead{C.E.~Arreche, A.~Bainbridge, B.~Obert, and A.~Ullah}
\newcommand{\glnc}{\mathrm{GL}_n(\mathbb{C})}
\newcommand{\gald}{\mathrm{Gal}_{\Delta}}
\newcommand{\gal}{\mathrm{Gal}}

\begin{document}

\title{Complex reflection groups as differential Galois groups}
\author{
Carlos E.~Arreche, Avery Bainbridge, Ben Obert, Alavi Ullah\\
The University of Texas at Dallas}

\date{}

\maketitle

\begin{abstract}
Complex reflection groups comprise a generalization of Weyl groups of semisimple Lie algebras, and even more generally of finite Coxeter groups. They have been heavily studied since their introduction and complete classification in the 1950s by Shephard and Todd, due to their many applications to combinatorics, representation theory, knot theory, and mathematical physics, to name a few examples. For each given complex reflection group G, we explain a new recipe for producing an integrable system of linear differential equations whose differential Galois group is precisely G. We exhibit these systems explicitly for many (low-rank) irreducible complex reflection groups in the Shephard-Todd classification.
\end{abstract}

\section{Introduction}
Classical Galois theory associates with a polynomial $p(y)\in K[y]$ with coefficients in a field $K$ a finite group $G$ of permutations of its roots, called the \emph{Galois group}, which encodes their algebraic properties. Thus via Galois theory the study of polynomial equations becomes a chapter in applied group theory. In the opposite direction, given a field $K$ and a finite group $G$, the \emph{inverse Galois problem} \cite{serre} for $(K,G)$ asks whether there exists a polynomial over $K$ whose Galois group is $G$ -- if so, the \emph{constructive} inverse Galois problem further asks to construct such a polynomial explicitly. An analogous differential Galois theory \cite{vdps} associates with an integrable linear differential system over a differential field $K$ a linear algebraic group that encodes the algebraic properties of the solutions to the system. Also analogously there is an inverse Galois problem that asks whether a given linear algebraic group $G$ can be realized as a differential Galois group over a given differential field $K$, and a constructive version that asks for the explicit construction of some such system, if one exists. An important result of Kolchin (see Theorem~\ref{thm:kolchin} below) states that finite Galois extensions of differential fields are special cases of differential Galois extensions. Crucially, the same finite group acts in classical Galois theory by permutations on a finite set of roots, and in differential Galois theory by linear transformations of a vector space of solutions. For $K=\mathbb{C}(z)$ endowed with the standard derivation $\frac{d}{dz}$, the non-constructive versions of the inverse Galois problem, both classical and differential, are known to have positive answers. It is still often considered interesting in this case to compute explicitly the corresponding (polynomial or differential) equations realizing a given group as a Galois group. 

Complex reflection groups comprise a very well-studied class of finite groups of matrices, which are characterized by having the best possible invariant theory (see \S\ref{sec:crg} below for more details). It was observed by Chevalley in \cite{chevalley} that these groups are realized as Galois groups over fields of multivariate rational functions. In many cases, explicit polynomial equations making this a \emph{constructive} solution to the inverse Galois problem are not known (although it is clear theoretically how to compute them using elimination theory and Gr\"obner bases). In any case, since complex reflection groups are given as finite groups of matrices from the start, it would be more natural to realize them explicitly as \emph{differential} Galois groups. Indeed, this is accomplished in the celebrated \cite{beukers}, where Beukers and Heckman list explicit generalized hypergeometric linear differential equations realizing complex reflection groups as differential Galois groups over the field of univariate rational functions $\mathbb{C}(z)$. Using entirely different methods, we develop an algorithm for constructing explicit integrable systems of linear differential equations over multivariate rational function fields $\mathbb{C}(z_1,\dots,z_n)$ realizing any given complex reflection group as a \mbox{differential Galois group.}

\section{Differential Galois Theory} \label{sec:dgt}
We refer to \cite[Appendix D.2]{vdps} and \cite[Chapter VI]{kolchin} for the material in this section.  A $\Delta$-field is a field $K$ equipped with a finite set $\Delta = \{\delta_1,\dots,\delta_n\}$ of pairwise-commuting derivations on $K$. We call $C_K:=\{c\in K \ | \ \delta_i(c)=0 \ \text{for} \ i=1,\dots,n\}$ the field of $\Delta$\emph{-constants} of $K$. A (rank $N$) \emph{linear differential system} over the $\Delta$-field $K$ is a collection of equations $\mathcal{A}=\bigl\{\delta_i(\mathbf{y})=A_i\mathbf{y} \ \big| \ 1\leq i \leq n\bigr\}$, where $\mathbf{y}=(y_1,\dots,y_N)^\top$ is a vector of unknowns, the $N\times N$ coefficient matrices $A_1,\dots, A_n$ have entries in $K$, and the $\delta_i$ are applied coordinate-wise. We say the system $\mathcal{A}$ is \emph{integrable} if $\delta_i(A_j)-\delta_j(A_i)=A_iA_j-A_jA_i$ for all $1\leq i,j\leq n$. A $\Delta$-field extension $L\supseteq K$ is a \emph{Picard-Vessiot extension} for the integrable linear differential system $\mathcal{A}$ if $C_L=C_K$ and $L$ is generated as a field extension of $K$ by the entries of a \emph{fundamental solution matrix} $U\in\mathrm{GL}_N(L)$ such that $\delta_i(U)=A_iU$ for each $i=1,\dots,n$. The existence and uniqueness (up to \mbox{$K$-$\Delta$-isomorphism}) of Picard-Vessiot extensions is guaranteed by the assumption that $C_K$ is algebraically closed of characteristic zero, which we impose from now on. In this case, the group of $K$-$\Delta$-automorphisms of $L$ over $K$ is the \emph{differential Galois group} $\gald(L/K)$, which is identified with a linear algebraic \mbox{group via}
\begin{equation} \label{eq:galois-rep}\gald(L/K)\hookrightarrow \mathrm{GL}_N(C_K):\gamma\mapsto U^{-1}\cdot\gamma(U)=:M_\gamma.\end{equation}
  A different choice of fundamental solution matrix $U$ for $\mathcal{A}$ results in a conjugate $C_K$-linear representation.

It is well-known that every (separable) algebraic field extension $L$ of $K$ is automatically a $\Delta$-field extension of $K$. It is relevant for us to recall the recipe. For any derivation $\delta$ acting on $K$ and any $\alpha\in L$ with minimal polynomial $p(y)\in K[y]$, there exists a unique derivation on $K(\alpha)$ extending $\delta$, defined by \begin{equation}
    \label{eq:delta-ext}
\delta(\alpha)=-p^\delta(\alpha)/p'(\alpha),\end{equation} where $p^\delta$ is obtained by applying $\delta$ to the coefficients, and $p'$ denotes the usual derivative of a polynomial. 

\begin{theorem}[Kolchin] \label{thm:kolchin} $L$ is a finite Galois extension of $K$ if and only if $L$ is a Picard-Vessiot extension of $K$ whose differential Galois group is finite, and moreover in this case $\gal(L/K)=\gald(L/K)$.
\end{theorem}

\section{Complex Reflection Groups} \label{sec:crg}

We refer to \cite{lehrer} for the material in this section. We say that $M \in \glnc$ is a \emph{complex reflection} if $M$ has finite order and $M$ fixes pointwise some codimension-one hyperplane. A finite group $G \subset \glnc$ generated by reflections is called a \emph{complex reflection group}. Such a $G$ acts on $S:=\mathbb{C}[x_1,\dots,x_n]$ by $\gamma_M(f)(\mathbf{x}):=f\bigl(\mathbf{x}\cdot M^{-\top}\bigr)$ with $\mathbf{x}:=(x_1,\dots,x_n)$ as in \cite[\S3.2]{lehrer}, and this action extends naturally to $L:=\mathbb{C}(x_1,\dots,x_n)$. We say $f \in S$ is \mbox{$G$\emph{-invariant}} if $\gamma_M( f) = f$ for every $M \in G$. We denote by $S^G\subseteq S$ the $\mathbb{C}$-subalgebra of $G$-invariant polynomials. The following result characterizes complex reflection groups among all finite groups of complex matrices as those having the best possible invariants \cite{chevalley,st}.
\begin{theorem}[Chevalley-Shephard-Todd] \label{thm:cst}
 A subgroup $G \subset \glnc$ is a complex reflection group if and only if $S^G=\mathbb{C}[\phi_1,\dots,\phi_n]$ for some homogeneous, $\mathbb{C}$-algebraically independent $G$-invariants $\phi_1(\mathbf{x}),\dots,\phi_n(\mathbf{x})$.
\end{theorem}

\noindent Such a set of $\mathbb{C}$-algebraically independent generators for $S^G$ is called a set of \emph{fundamental invariants} for $G$. Any two sets of fundamental invariants have the same multiset of total degrees \cite[Prop.~3.25]{lehrer}. Having chosen an $n$-tuple $\boldsymbol{\phi}=(\phi_1,\dots,\phi_n)$ of fundamental invariants, let us denote the \emph{Jacobian matrix} \begin{equation} \label{eq:jac} J_{\boldsymbol{\phi}}:=\begin{pmatrix} \tfrac{\partial\phi_1}{\partial x_1} & \cdots & \tfrac{\partial\phi_1}{\partial x_n}\\ \vdots & & \vdots \\ \tfrac{\partial\phi_n}{\partial x_1} & \cdots & \tfrac{\partial\phi_n}{\partial x_n} \end{pmatrix}.\end{equation} The following result is proved in \cite[Lem.~6.6]{lehrer}.

\begin{lemma} \label{lem:jac}
    Suppose $G\subset \glnc$ is a complex reflection group and $\boldsymbol{\phi}$ is an $n$-tuple of fundamental invariants. Then $\gamma_M(J_{\boldsymbol{\phi}})=J_{\boldsymbol{\phi}}\cdot M$.
\end{lemma}

\section{Goal and Main Result}

Suppose $G\subset\mathrm{GL}_n(\mathbb{C})$ is a complex reflection group. By Theorem~\ref{thm:cst}, we can identify $S^G$ with a free polynomial algebra $\mathbb{C}[z_1,\dots,z_n]$ (having chosen an $n$-tuple of fundamental invariants $\boldsymbol{\phi}\leftrightarrow \mathbf{z}$). As proved in \cite{chevalley}, $L^G=\mathbb{C}(z_1,\dots,z_n)=:K$. By Artin's Theorem \cite[Thm.~VI.1.8]{lang}, $L$ is a Galois extension of $K$ with Galois group $G\xrightarrow{\sim}\gal(L/K)$ via $M\mapsto \gamma_M$. We consider $K$ as a $\Delta$-field with derivations $\delta_i=\frac{\partial}{\partial z_i}$ for $i=1,\dots, n$. Then $L$ is a Picard-Vessiot extension of $K$ with $\gald(L/K)\simeq G$ by Theorem~\ref{thm:kolchin}.

\begin{goal}
    Compute explicitly an integrable linear differential system $\mathcal{A}$ over $K$ such that $L$ is the Picard-Vessiot extension of $K$ for $\mathcal{A}$ and whose differential Galois group is $G\subset\mathrm{GL}_n(\mathbb{C})$ acting as in \eqref{eq:galois-rep}.
\end{goal}

 This is the sense in which we wish to eponymously realize complex reflection groups as differential Galois groups. A first obstruction to doing this is that, although we know that the coordinate derivations $\delta_i$ on $K$ extend uniquely to derivations on $L$, the formula \eqref{eq:delta-ext} presumes we have already computed a separable polynomial satisfied by $x_1,\dots,x_n\in L$, which is a computation that we are hoping to avoid. Thus we begin by providing an alternative computation of the action of $\Delta$ on $L$, as follows. Denoting $\eta_{ij}:=\delta_j(x_i)\in L$, we see immediately that the operators
 $\delta_j=\sum_{i=1}^n \eta_{ij}\frac{\partial}{\partial x_i}$.
One shows using standard differential geometry arguments that the missing coefficients $\eta_{ij}\in L$ are computed by $J_{\boldsymbol{\phi}}^{-1} = (\eta_{ij})$, for the same Jacobian matrix \eqref{eq:jac}. Now that we can differentiate arbitrary elements of $L$ with respect to the coordinate derivations in $\Delta$, we can state and prove our main result, which accomplishes our aforementioned Goal.

\begin{theorem}\label{thm:main} Denote $A_i := \delta_i(J_{\boldsymbol{\phi}}) J_{\boldsymbol{\phi}}^{-1}\in\mathfrak{gl}_n(L)$ for $1\leq i \leq n$. Then:  
\begin{enumerate}
\item $A_i\in\mathfrak{gl}_n(K)$ for each $1\leq i \leq n$;
\item $\delta_i(A_j)-\delta_j(A_i)=A_iA_j-A_jA_i$ for $1\leq i,j\leq n$; and
\item $L$ is a Picard-Vessiot extension of $K$ for the system $\mathcal{A}=\bigl\{\delta_i(\mathbf{y})=A_i\mathbf{y} \ \big| \ 1\leq i \leq n\bigr\}$.
\end{enumerate}
\end{theorem}

\begin{proof}[Proof sketch.] (1). The $M\in G$ act on $L$ by $\Delta$-automorphisms $\gamma_M$. By Lemma~\ref{lem:jac}, $\gamma_M(J_{\boldsymbol{\phi}})=J_{\boldsymbol{\phi}}\cdot M$. Hence each $\gamma_M(A_i)=A_i$ for every $M\in G$. By the Galois correspondence, the $A_i$ must have entries in $K$.

(2). This is a familiar computation: using that $\delta_i\delta_j=\delta_j\delta_i$ on $L$, expand $\delta_i(\delta_j(J_{\boldsymbol{\phi}}))=\delta_j(\delta_i(J_{\boldsymbol{\phi}}))$.

(3). Since $\gamma_M\mapsto J_{\boldsymbol{\phi}}^{-1}\gamma_M(J_{\boldsymbol{\phi}})=M$ is injective, $G$ acts faithfully on $\tilde{L}:=K(J_{\boldsymbol{\phi}})\subseteq L$, and therefore $\tilde{L}=L$ by the Galois correspondence.\qedhere
\end{proof}

\section{A Small Dihedral Example}
The group $D_8=\left\{\left(\begin{smallmatrix}    \pm 1 & 0 \\ 0 & \pm 1  \end{smallmatrix}\right), \left(\begin{smallmatrix}   0 & \pm 1 \\ \pm 1 & 0 \end{smallmatrix}\right) \right\}$ of symmetries of the square, denoted $G(2,1,2)$ in the notation of \cite{st}, is generated by the reflections $M_1=\left(\begin{smallmatrix} 1 & 0 \\ 0 & -1\end{smallmatrix}\right)$ and $M_2=\left(\begin{smallmatrix} 0 & 1 \\ 1 & 0\end{smallmatrix}\right)$. It acts on polynomials in $S=\mathbb{C}[x_1,x_2]$ by $\gamma_{M_1}(f)(x_1,x_2)=f(x_1,-x_2)$ and $\gamma_{M_2}(f)(x_1,x_2)=f(x_2,x_1)$.
The algebra of invariants is $S^G=\mathbb{C}[z_1,z_2]$ with $z_1=\phi_1(x_1,x_2)=x_1^2+x_2^2$ and $z_2=\phi_2(x_1,x_2)=x_1^2x_2^2$. For these fundamental invariants, the Jacobian 
  \[ J=\left(\begin{smallmatrix} 2x_1 \hphantom{X}& 2x_2 \\ 2x_1x_2^2 \hphantom{X}& 2x_1^2x_2\end{smallmatrix}\right);\quad\qquad\text{and its inverse}\quad\qquad J^{-1}=\tfrac{1}{2x_1x_2(x_1^2-x_2^2)}\left(\begin{smallmatrix} x_1^2x_2 \hphantom{X}& -x_2 \\ -x_1x_2^2 \hphantom{X}& x_1\end{smallmatrix}\right). \]
  Thus the derivations $\frac{\partial}{\partial z_1}=\delta_1$ and $\frac{\partial}{\partial z_2}=\delta_2$ acting on $S=\mathbb{C}[x_1,x_2]$ are given by
  \[\delta_{1}=\tfrac{x_1}{2(x_1^2-x_2^2)}\cdot\tfrac{\partial}{\partial x_1} - \tfrac{x_2}{2(x_1^2-x_2^2)}\cdot\tfrac{\partial}{\partial x_2}\qquad\quad \text{and}\qquad \quad\delta_{2}=\tfrac{1}{2x_2(x_1^2-x_2^2)}\cdot\tfrac{\partial}{\partial x_2} - \tfrac{1}{2x_1(x_1^2-x_2^2)}\cdot\tfrac{\partial}{\partial x_1},\] which can be checked explicitly from the symbolic expressions $x_1=\sqrt{\frac{z_1\pm\sqrt{z_1^2-4z_2}}{2}}$ and $x_2=\sqrt{\frac{z_1\mp\sqrt{z_1^2-4z_2}}{2}}$. Differentiating entrywise, $\delta_{1}(J)=\tfrac{1}{x_1^2-x_2^2}\left(\begin{smallmatrix} x_1 \hphantom{X} & -x_2 \\ -x_1x_2^2\hphantom{X} & x_1^2x_2\end{smallmatrix}\right)$ and $\delta_{2}(J)=\tfrac{1}{x_1^2-x_2^2}\left(\begin{smallmatrix} -x_1^{-1} \hphantom{X}& x_2^{-1} \\ 2x_1-x_1^{-1}x_2^2 \hphantom{X} & x_1^2x_2^{-1}-2x_2\end{smallmatrix}\right)$, 
whence the $A_1:=\delta_{1}(J)J^{-1}$ and $A_2:=\delta_{2}(J)J^{-1}$ of Theorem~\ref{thm:main} are initially computed as \[ A_1=\tfrac{1}{2x_1x_2(x_1^2-x_2^2)^2}\left(\begin{smallmatrix}x_1^3x_2 +x_1x_2^3\hphantom{X} & -2x_1x_2 \\ -2x_1^3x_2^3 \hphantom{X}& x_1^3x_2+x_2x_2^3\end{smallmatrix}\right) \qquad\!\text{and}\qquad\!
    A_2=\tfrac{1}{2x_1x_2(x_1^2-x_2^2)^2}\left(\begin{smallmatrix}-2x_1x_2 \hphantom{X}& x_1x_2^{-1} + x_1^{-1}x_2 \\ x_1^3x_2+x_1x_2^3 \hphantom{X}& x_1^3x_2^{-1}-4x_1x_2+x_1^{-1}x_2^3\end{smallmatrix}\right).\]
According to Theorem~\ref{thm:main}, the entries of $A_1$ and $A_2$ should actually be rational in $z_1$ and $z_2$, and not merely in $x_1$ and $x_2$. Our software carries out the necessary rewriting to discover automatically that indeed \[A_1=\tfrac{1}{2(z_1^2-4z_2)}\Bigl(\begin{smallmatrix} z_1 & -2 \\ -2z_2 & z_1   \end{smallmatrix}\Bigr) \qquad \text{and} \qquad A_2=\tfrac{1}{2(z_1^2-4z_2)}\left(\begin{smallmatrix} -2 & z_1z_2^{-1} \\ z_1 & z_1^2z_2^{-1}-6   \end{smallmatrix}\right).\]

\section{Algorithmic Considerations}
As we saw in the small dihedral example above, the invariant entries of the matrices $A_1,\dots,A_n$ from Theorem~\ref{thm:main} are initially expressed as elements of $L=\mathbb{C}(\mathbf{x})$, which leads us to the \emph{rewriting problem} of expressing them as elements of $K=\mathbb{C}(\mathbf{z})$. First let us explain how to reduce this rewriting problem for rational functions to the simpler rewriting problem for homogeneous polynomials. Denoting by $d_i$ the homogeneous degree of $\phi_i(\mathbf{x})$, the entries along the $i$-th row of the Jacobian $J_{\boldsymbol{\phi}}$ are all homogeneous of degree $d_i-1$. Hence $\mathrm{det}(J_{\boldsymbol{\phi}})$ is homogeneous of degree $\sum_i=(d_i-1)$, and moreover the entries along the $j$-th column of the adjugate $\mathrm{det}(J_{\boldsymbol{\phi}})J_{\boldsymbol{\phi}}^{-1}$ are all homogeneous polynomials of degree $\sum_{k\neq n-j}(d_k-1)$. Thus the $i$-th row of $\mathrm{det}(J_{\boldsymbol{\phi}})\delta_\ell(J_{\boldsymbol{\phi}})$ consists of homogeneous polynomials of degree $d_i-2+\sum_{k\neq n-\ell}(d_k-1)$. Hence, the $(i,j)$-entry of $\mathrm{det}(J_{\boldsymbol{\phi}})^2 A_\ell$ is a homogeneous polynomial of degree $d_i-2+\sum_{k\neq n-\ell}(d_k-1)+\sum_{k\neq n-j}(d_k-1)$. By \cite[Cor.~6.7]{st}, $\mathrm{det}(J_{\boldsymbol{\phi}})^m$ is invariant for some $m\geq 2$. Thus our rational rewriting problem is reduced to the rewriting problem for the invariant homogeneous \emph{polynomial} entries of the matrices $\mathrm{det}(J_{\boldsymbol{\phi}})^{m}A_\ell$.

We know by Theorem~\ref{thm:cst} that if (and only if) a given polynomial $f(\mathbf{x}) \in \mathbb{C}[\mathbf{x}]$ is invariant, there exists a unique $\tilde{f}(\mathbf{z}) \in \mathbb{C}[\mathbf{z}]$ such that $\tilde{f}(\boldsymbol{\phi}) = f(\mathbf{x})$. Finding this $\tilde{f}$ is reduced to linear algebra. Indeed, denoting by $E_f:=\bigl\{(e_1,\dots,e_n)\in\mathbb{Z}^n_{\geq 0} \ \big| \ \sum_{i=1}^ne_id_i=\mathrm{deg}_\mathbf{x}(f)\bigr\}$, there exist unique $c_{\mathbf{e}}\in\mathbb{C}$ for $\mathbf{e}\in E_f$ such that $f(\mathbf{x})=\sum_{\mathbf{e}\in E_f}c_\mathbf{e}\prod_{i=1}^n\phi_i(\mathbf{x})^{e_i}$. Solving the resulting linear system, we obtain $\tilde{f}(\mathbf{z})=\sum_{\mathbf{e}}c_{\mathbf{e}}\mathbf{z}^{\mathbf{e}}$. Using a preliminary Mathematica \cite{mathematica} implementation of this recipe, we have successfully computed explicitly the integrable systems for primitive complex reflection groups of ranks $2$ and $3$. 

\section{Results for the Tetrahedral Groups} Below is the output of our algorithm for the first four primitive complex reflection in the classification of \cite{st}, with fundamental invariants computed as in \cite[\S6.6]{lehrer}.
\[\footnotesize\begin{array}{|c|l|c|c|}
\hline
\text{Group} & \text{Fundamental Invariants} & A_1 & A_2 \\ \hline
G_4 \vphantom{\begin{matrix} X \\ X \\ X \\ X \end{matrix}} & 
\begin{aligned}
    z_1 &:= x_1^4 + 2i \sqrt{3} x_1^2 x_2^2 + x_2^4\\ z_2 &:= x_1^5 x_2 - x_1 x_2^5\end{aligned}
&
\begin{pmatrix}  \frac{3 z_1^2}{4 \left(z_1^3-12 i \sqrt{3} z_2^2\right)} & -\frac{6 i \sqrt{3} z_2}{z_1^3-12 i \sqrt{3} z_2^2} \\
-\frac{15 z_1 z_2}{8 \left(z_1^3-12 i \sqrt{3} z_2^2\right)} & \frac{5 z_1^2}{4 \left(z_1^3-12 i \sqrt{3} z_2^2\right)}\end{pmatrix}
&
\begin{pmatrix}  -\frac{6 i \sqrt{3} z_2}{z_1^3-12 i \sqrt{3} z_2^2} & \frac{4 i \sqrt{3} z_1}{z_1^3-12 i \sqrt{3} z_2^2} \\
 \frac{5 z_1^2}{4 \left(z_1^3-12 i \sqrt{3} z_2^2\right)} & -\frac{10 i \sqrt{3} z_2}{z_1^3-12 i \sqrt{3} z_2^2}\end{pmatrix}
\\ \hline
G_5 \vphantom{\begin{matrix} X \\ X \\ X \\ X \end{matrix}} & \begin{aligned}
    z_1 &:= \bigl(x_1^4 + 2i \sqrt{3} x_1^2 x_2^2 + x_2^4\bigr)^3\\ z_2 &:= x_1^5 x_2 - x_1 x_2^5\end{aligned} &
\begin{pmatrix} \frac{11 z_1-96 i \sqrt{3} z_2^2}{12 z_1 \left(z_1-12 i \sqrt{3} z_2^2\right)} & -\frac{6 i \sqrt{3} z_2}{z_1-12 i \sqrt{3} z_2^2} \\ -\frac{5 z_2}{24 z_1 \left(z_1-12 i \sqrt{3} z_2^2\right)} & -\frac{5}{12 \left(z_1-12 i \sqrt{3} z_2^2\right)}\end{pmatrix}
&
\begin{pmatrix}  -\frac{6 i \sqrt{3} z_2}{z_1-12 i \sqrt{3} z_2^2} & \frac{12 i \sqrt{3} z_1}{z_1-12 i \sqrt{3} z_2^2} \\
 -\frac{5}{12 \left(z_1-12 i \sqrt{3} z_2^2\right)} & -\frac{10 i \sqrt{3} z_2}{z_1-12 i \sqrt{3} z_2^2} \end{pmatrix}
  \\ \hline
G_6 \vphantom{\begin{matrix} X \\ X \\ X \\ X \end{matrix}} & \begin{aligned}
    z_1 &:= x_1^4 + 2i \sqrt{3} x_1^2 x_2^2 + x_2^4\\ z_2 &:= \bigl(x_1^5 x_2 - x_1 x_2^5\bigr)^2\end{aligned} &
\begin{pmatrix} \frac{3 z_1^2}{4 \left(z_1^3-12 i \sqrt{3} z_2\right)} & \frac{3 i \sqrt{3}}{z_1^3-12 i \sqrt{3} z_2} \\
-\frac{15 z_1 z_2}{4 \left(z_1^3-12 i \sqrt{3} z_2\right)} & \frac{5 z_1^2}{4 \left(z_1^3-12 i \sqrt{3} z_2\right)}\end{pmatrix}
 &
 \begin{pmatrix}   \frac{3 i \sqrt{3}}{z_1^3-12 i \sqrt{3} z_2} & \frac{i \sqrt{3} z_1}{z_2 \left(z_1^3-12 i \sqrt{3} z_2\right)} \\
     \frac{5 z_1^2}{4 \left(z_1^3-12 i \sqrt{3} z_2\right)} & \frac{z_1^3-22 i \sqrt{3} z_2}{2 z_2 \left(z_1^3-12 i \sqrt{3} z_2\right)} \end{pmatrix}
 \\ \hline
G_7 \vphantom{\begin{matrix} X \\ X \\ X \\ X \end{matrix}} & \begin{aligned}
    z_1 &:= \bigl(x_1^4 + 2i \sqrt{3} x_1^2 x_2^2 + x_2^4\bigr)^3\\ z_2 &:= \bigl(x_1^5 x_2 - x_1 x_2^5\bigr)^2\end{aligned} &
\begin{pmatrix}  \frac{11 z_1-96 i \sqrt{3} z_2}{12 z_1 \left(z_1-12 i \sqrt{3} z_2\right)} & \frac{3 i \sqrt{3}}{z_1-12 i \sqrt{3} z_2} \\
     -\frac{5 z_2}{12 z_1 \left(z_1-12 i \sqrt{3} z_2\right)} & -\frac{5}{12 \left(z_1-12 i \sqrt{3} z_2\right)} \end{pmatrix}
&
\begin{pmatrix} \frac{3 i \sqrt{3}}{z_1-12 i \sqrt{3} z_2} & \frac{3 i \sqrt{3} z_1}{z_2 \left(z_1-12 i \sqrt{3} z_2\right)} \\
-\frac{5}{12 \left(z_1-12 i \sqrt{3} z_2\right)} & \frac{z_1-22 i \sqrt{3} z_2}{2 z_2 \left(z_1-12 i \sqrt{3} z_2\right)} \end{pmatrix}
 \\ \hline
\end{array}\]
The primitive complex reflection groups in rank $2$ of octahedral and icosahedral type all produce similarly compact and tractable outputs. In rank~$3$, some of the outputs look considerably more complicated.

\end{document}